%% file: NoteOnHilbert_arxiv2.tex
\documentclass[reqno]{amsart}

\usepackage{amsmath}
\usepackage{amssymb}
\usepackage[linktocpage=true]{hyperref}
\hypersetup{
    colorlinks=true,
    linkcolor=blue,
    filecolor=magenta,      
    urlcolor=cyan,
    pdftitle={Overleaf Example},
    pdfpagemode=FullScreen,
    }
\usepackage[hypcap=true]{caption}
\usepackage{enumitem}

\theoremstyle{plain}
\newtheorem{theorem}{Theorem}[section]
\newtheorem{lemma}[theorem]{Lemma}
\newtheorem{proposition}[theorem]{Proposition}

\theoremstyle{definition}

\newtheorem{remark}[theorem]{Remark}

\theoremstyle{plain}
\newtheorem{ltheorem}{Theorem}

\newcommand{\extR}{\widehat{\mathbb{R}}}
\newcommand{\extC}{\widehat{\mathbb
{C}}}
\newcommand{\PSL}{\mathrm{PSL}}
\renewcommand{\Re}{\mathrm{Re}}
\renewcommand{\Im}{\mathrm{Im}}
\newcommand{\SL}{\mathrm{SL}}
\newcommand{\PGL}{\mathrm{PGL}}
\newcommand{\PSU}{\mathrm{PSU}}
\renewcommand{\L}{\mathcal{L}}
\renewcommand{\O}{\mathcal{O}}
\newcommand{\nBianchi}{\Gamma_n}
\newcommand\restr[2]{{
  \left.\kern-\nulldelimiterspace 
  #1 
  \vphantom{\big|} 
  \right|_{#2} 
  }}

\makeatletter
\newcommand*{\rom}[1]{\expandafter\@slowromancap\romannumeral #1@}
\makeatother

\begin{document}
\title{A note on
  Hilbert transform over lattices of \texorpdfstring{$\mathrm{PSL}_2(\mathbb{C})$}{PSL2C}
}

\author{
  Jorge Pérez García
  }
\address{Instituto de Ciencias Matemáticas (ICMAT), C. Nicolás Cabrera, 13-15, 13-15, \linebreak Fuencarral-El Pardo, 28049 Madrid, Spain}
\email{jorge.perez@icmat.es}

\begin{abstract}
  \input{abstract}
\end{abstract}



\maketitle

\section*{Introduction}

The boundedness problem for Fourier multipliers on $L_p$-spaces has always played a central role in harmonic analysis. One of the most studied examples is the Hilbert transform, defined as $\widehat{Hf}(\xi) = i \,\mathrm{sign}(\xi) \widehat{f}(\xi)$ for $f\in L_2(\mathbb{R})$. Although $H$ was already known to be bounded in $L_p(\mathbb{R})$ for $1<p<\infty$, in 1955 Cotlar \cite{Cotlar1955Unified} gave a very simple proof of this fact using the following identity:
\begin{equation}\tag{Classical Cotlar}\label{eq:classical Cotlar}
     |Hf|^2=2 H(f \cdot Hf)-H(H(|f|^2)).
\end{equation}
This is known nowadays as the Cotlar identity. His proof uses that $H$ is bounded in $L_2(\mathbb{R})$ and that, by a recursive use of \eqref{eq:classical Cotlar}, it also must be bounded in every $p=2^k$ for $k\geq 1$. Interpolation and the fact that $H$ is self-adjoint complete the proof.

Mei and Ricard \cite{MeiRicard2017FreeHilb} introduced the Cotlar identity in the non-commutative setting in order to study Hilbert transforms over free groups and amalgamated free products of von Neumann algebras. In the recent work of Gonz\'alez-P\'erez, Parcet and Xia \cite{GonParXia2022} the authors developed a systematic approach to study Cotlar identities for Fourier multipliers in non-Abelian groups. 
Let $G$ be an unimodular group, $\L G$ the von Neumann algebra of $G$ and $G_0\subset G$ an open subgroup. 
Consider $m\colon G\to \mathbb{C}$ a symbol on $G$ and $T_m$ the corresponding Fourier multiplier on $\L G$. Then the formula:
\begin{equation}\tag{Cotlar}\label{eq:Cotlar for symbol}
    (m(g^{-1}) - m(h))(m(gh) - m(g)) = 0, \quad \text{for all } g\in G\setminus G_0,\,h\in G,
\end{equation}
is a translation of \eqref{eq:classical Cotlar} for $T_m$ in terms of its symbol. The main result in \cite{GonParXia2022} states that any $m$ which is bounded, left $G_0$-invariant and verifies \eqref{eq:Cotlar for symbol} defines a bounded multiplier in $L_p(\L\,G)$ for all $1<p<\infty$.

The subgroup $G_0$ represents a set in which the Cotlar identity may fail. In the argument, this failure is balanced by the invariance of $m$ with respect to $G_0$. Therefore this formulation of the theorem allows more flexibility in terms of the multiplier than the original one. However, the hypothesis of invariance can be relaxed even further. If $\chi\colon G_0\to \mathbb{T}^1$ is a character, it is enough for the result to hold that $m$ verifies:
\begin{equation*}
    m(gh) = \chi(g)m(h)\quad \text{for all }g\in G_0, \, h\in G.
\end{equation*}
 We say in this case that $m$ is \textit{left $(G_0,\chi)$-equivariant,} and of course the $G_0$-invariance is recovered when $\chi$ is the trivial character.

\subsection*{Hilbert transform in \texorpdfstring{$\PSL_2(\mathbb{C})$}{PSL2(C)}}
Recall that $\PSL_2(\mathbb{C})$, which is the quotient of the $2\times 2$ complex matrices with determinant $1$ by its center, can be identified with the group of orientation-preserving isometries of the three dimensional hyperbolic space $\mathbb{H}^3$. This identification can be made explicit in various ways. Here we give one using the upper-space model of $\mathbb{H}^3$ and quaternions. Let $i,j,k$ denote the usual three quaternionic units, and let's define:
\begin{equation*}
    \mathbb{H}^3 =
    \{x+yi+rj \colon x,y,r\in \mathbb{R},\, r>0\}.
\end{equation*}
Doing so, $\mathbb{H}^3$ is exactly the subspace $\mathbb{C}+\mathbb{R}_{>0}j$ of the quaternions. 
Now, for a given $\omega \in \mathbb{H}^3$ we set:
\begin{equation*}
    g\cdot \omega = 
    (a\omega + b)(c\omega+d)^{-1},
    \quad \text{for }
    g =
    \begin{bmatrix}
        a & b \\
        c & d
    \end{bmatrix}
    \in \PSL_2(\mathbb{C}).
\end{equation*}
It is possible to compute the inverse of a quaternion using its conjugate and modulus. This leads to a more explicit formula for the action of $g\in\PSL_2(\mathbb{C})$ on the element $\omega = z+rj\in \mathbb{C}+\mathbb{R}_{>0}j$, namely:
$$ g\cdot \omega = \frac{a\overline{c}|z+rj|^2 + b\overline{d} + a\overline{d}z + b\overline{c}\overline{z} + rj}{|c(z+rj)+ d|^2}
$$
This is a well-defined action of $\PSL_2(\mathbb{C})$ on $\mathbb{H}^3$. Indeed, $\PSL_2(\mathbb{C})$ acts by orientation-preserving isometries on $\mathbb{H}^3$ when equipped with the usual Riemannian metric:
$$ ds^2 = \frac{dx^2 + dy^2 + dr^2}{r^2},
$$
and it is the full group of such isometries (see \cite{Ahl1985} for more details).

On the other hand, a group $G$ acting on a set $X$ induces a multiplier on $G$ as follows: first choose a point $x_0\in X$ and two disjoint subsets $X^+, X^- \subset X$. Let $m$ be the map $m\colon G\to \mathbb{C}$ defined for each $g\in G$ as:
\begin{equation*}
m(g)=
\begin{cases}
    1 & \text{if } g\cdot x_0\in X^+,\\
    -1 & \text{if } g\cdot x_0\in X^-,\\
    0 & \text{otherwise}.    
\end{cases}
\end{equation*}

Even if the final multiplier depends on $x_0$ and also on the partition given by $X^+$ and $X^-$, the boundedness of the multiplier is preserved by changing $x_0$ for any other point in the same $G$-orbit or using the sets $\{g\cdot X^+, g\cdot X^-\}$, with $g\in G$, instead of $\{X^+, X^-\}$.
Back to the action of $\PSL_2(\mathbb{C})$ on the hyperbolic space, we are choosing the base point in $\mathbb{H}^3$ given by $j$ in our quaternionic parametrization, and the following partition:
\begin{align*}
    \Sigma^+ = \{\omega \in \mathbb{H}^3\colon \Re(\omega)>0 \},
    \quad \text{and}\quad
    \Sigma^- = \{\omega \in \mathbb{H}^3\colon \Re(\omega)<0 \}.
\end{align*}
This procedure induces a symbol $m$ in $\PSL_2(\mathbb{C})$ that is explicitly given by:
\begin{equation}\tag{1}\label{eq:def of m}
    m(g) = \mathrm{sign}( \Re(a\overline{c} + b\overline{d})),\quad \text{with }
    g = \begin{bmatrix}
        a & b \\
        c & d
    \end{bmatrix} \in \PSL_2(\mathbb{C}).
\end{equation}
The dividing frontier $\Sigma = \mathbb{H}^3 \smallsetminus (\Sigma^+\cup \Sigma^-)= \{\omega \in \mathbb{H}^3\colon \Re(\omega)=0 \}$ is a hyperbolic plane, which determines the symbol $m$ up to a sign. Since the action of $\PSL_2(\mathbb{C})$ is transitive both on points and hyperbolic planes in $\mathbb{H}^3$, the boundedness of the multiplier defined by $m$ on $L_p(\L\, \PSL_2(\mathbb{C}))$ will remain the same under any other choice of the kind. Also, it is worth noticing that $m$ is easily shown  invariant under the action of two groups:
\begin{enumerate}[label=\roman*.]
    \item The right action of the group $\PSU(2)$, which is the image of the unitary group $\mathrm{U}(2)$ under the projection $\SL_2(\mathbb{C})\to \PSL_2(\mathbb{C})$.
    \item The left action of the group $G_0\leq \PSL_2(\mathbb{C})$ defined by:
    \begin{equation}\tag{2}\label{eq: def of G0}
        G_0 = \left\{ \begin{bmatrix}
            x & iy \\
            iz & w
        \end{bmatrix} \colon x,y,z,w\in \mathbb{R},\, xw+yz = 1
        \right\}.
    \end{equation}
\end{enumerate}

In \cite{GonParXia2022} the authors proved that, when restricted to the lattices $\PSL_2(\mathbb{Z})$ and $\PSL_2(\mathbb{Z}[i])$, this function defines an $L_p$-bounded Fourier multiplier for every $1<p<\infty$. They posed three related questions, namely:
\begin{enumerate}[label=\roman*.]
    \item Is this multiplier bounded in $L_p(\L\,\PSL_2(\mathbb{C}))$?
    \item Is its restriction bounded in $L_p(\L\,\PSL_2(\mathbb{R}))$?
    \item Are there more lattices $\Gamma\leq\PSL_2(\mathbb{C})$ for which the restriction of $m$ still defines a multiplier bounded in $L_p(\L\,\Gamma)$?
\end{enumerate}
The two first questions are negatively answered by the work of Parcet, de la Salle and Tablate. Concretely, by \cite[Corollary B2]{PardlSTab2023} and the fact that the Lie algebra of $\PSL_2(\mathbb{C})$ is simple (as a real Lie algebra) solves the problem.

In the present work we tackle the third question. Our main result concerns the family of groups $\nBianchi = \PSL_2(\mathbb{Z}[\sqrt{-n}])$, and it can be stated follows:
\begin{ltheorem}\label{th:main theorem}
    For any integer $n>0$, the symbol $m$ restricted to the group $\nBianchi$ defines a bounded Fourier multiplier in $L_p(\L\,\nBianchi)$ for all $1<p<\infty$, whose norm satisfies:
    $$\|T_m\colon L_p(\L\,\nBianchi)\to L_p(\L\,\nBianchi)\|\lesssim\left(\frac{p^2}{p-1}\right)^\beta,\quad \text{where }\beta = 1+\log_2(1+\sqrt{2}).$$
\end{ltheorem}

The proof consist in identifying a subgroup $K_n\leq \Gamma_n$ and a suitable character $\chi\colon K_n\to \mathbb{T}^1$ for which $m$ is left $(K_n,\chi)$-equivariant, and then proving by hand that \eqref{eq:Cotlar for symbol} holds.
Using results and ideas from \cite{Sta2017}, we refined the argument in \cite{GonParXia2022} for the case $n=1$, where the authors defined an auxiliary symbol $\widetilde m$ that is indeed $K_1$-invariant, and carried out the analogous computations that we present here in more generality. 

Bianchi groups are another natural family of lattices in $\PSL_2(\mathbb{C})$ to consider, which were introduced by Bianchi in \cite{bianchi1892sui} as a generalization of the group $\PSL_2(\mathbb{Z})$. For every square-free positive integer $n>0$, we define the $n$-th Bianchi group as $\Gamma_n' = \PSL_2(\O_{-n})$, where $\O_{-n}$ denotes the ring of integers of the quadratic extension $\mathbb{Q}(\sqrt{-n})$. The explicit definition of $\Gamma_n'$ depends on the class of $n$ modulus 4, since:
\begin{equation*}
    \O_{-n} =
\begin{cases}
    \mathbb{Z}[\sqrt{-n}] & \text{if } n\not\equiv -1 \, (\operatorname{mod} 4), \\
    \mathbb{Z}\left[\frac{1+\sqrt{-n}}{2}\right] & \text{otherwise}. 
\end{cases}
\end{equation*}
Therefore this family extends the one featuring in Theorem \ref{th:main theorem} when $n\equiv -1 \, (\operatorname{mod} 4)$. In this case, the problem is that the set where \eqref{eq:Cotlar for symbol} fails is bigger in $\Gamma_n'$ than in $\Gamma_n$. This set cannot be contained in a subgroup with respect to which $m$ has some kind of invariance, and this is why the Cotlar identity cannot hold in every Bianchi group $\Gamma_n'$ with $n\equiv -1 \, (\operatorname{mod} 4)$. 

The question of whether $m$ defines a bounded multiplier on $L_p(\L\,\Gamma_n')$ is left open in this case, but we are still able to prove that, choosing a different hyperbolic plane $\Sigma$ to induce our multiplier, one can still get symbols that satisfy Cotlar identity in most of Bianchi groups (all but $n=3$). Moreover, this approach also allows us to prove the Cotlar identity for the original $m$ in $\mathrm{PSL}_2(\mathbb{Z}[i])$ in a much simpler way than any of the previous proofs.

\section{Background}\label{sec: background}

\subsection*{Group von Neumann algebras}
Let $G$ be a discrete group and let $\lambda\colon G \to B(\ell_2(G))$ denote the left regular representation of $G$, that is, the unitary representation of $G$ assigning to every $g\in G$ the operator $\lambda_g\in B(\ell_2(G))$ given by $\lambda_gf(h) = f(g^{-1}h)$, for every $f\in \ell_2(G)$ and $h\in G$. The group von Neumann algebra of $G$, denoted here by $\L\,G$, is the operator algebra given by:
\begin{equation*}
    \L\,G = \overline{\textrm{span}\{\lambda_g \colon g\in G\}}^{\textrm{WOT}},
\end{equation*}
where closure is taken in the weak operator topology of $B(\ell_2(G))$.
Notice that an arbitrary element $x\in \L\,G$ can be represented by a sum $ x = \sum_{g\in G} x_g \lambda_g$, with $x_g\in \mathbb{C}$.

The group von Neumann algebra $\L\,G$ comes equipped with a finite trace:
\begin{equation*}
    \tau \colon \L\,G \to \mathbb{C}, \quad
    x\mapsto \tau \left(\sum_{g\in G} x_g \lambda_g\right) = x_e.
\end{equation*}
If $G$ is Abelian then $\L\,G$ is isomorphic (as von Neumann algebra) to $L_\infty(\widehat{G})$, where $\widehat{G}$ represents the dual group, and $\tau$ is the functional induced on $L_\infty(\widehat{G})$ by the Haar measure of $\widehat{G}$. 
In the non-commutative case, the trace $\tau$ above defined helps us to define $L_p$-spaces associated to $\L\,G$ without needing an underlying measure space. For a given $x\in \L\,G$ and $p\in [1,\infty]$ we define the norms:
\begin{equation*}
    \|x\|_p = \tau(|x|^p)^{1/p}\, \text{ if } 1\leq p <\infty, \text{ and } \,
    \|x\|_\infty = \|x\|_{\L\,G}.
\end{equation*}
The space $L_p(\L\,G)$ is defined as the completion of $B(\ell_2(G))$ with respect to this norm. All of this can be done in more generality for non-discrete groups, using the Haar measure of $G$ and defining a weight $\tau$ instead of a trace, see \cite{Ped1979}. The $L_p$-spaces over von Neumann algebras can also be defined in more generality, see for example \cite{PiXu2003}.

\subsection*{Non-commutative Fourier multipliers}
A Fourier multiplier $T_m$ with symbol $m\colon G\to \mathbb{C}$ is an operator defined as:
\begin{equation*}
    T_m\left(\sum_{g\in G} x_g\lambda_g\right) = \sum_{g \in G} m(g)x_g\lambda_g, \quad \text{for } x = \sum_{g\in G} x_g\lambda_g \in \mathbb{C}G.
\end{equation*}
Here $\mathbb{C}G$ denotes the space of elements with finite Fourier expansion. Notice that it is a dense subspace of every $L_p(\L\,G)$ for $1\leq p<\infty$. If $T_m$ extends to a bounded operator $T_m\colon L_p(\L\,G)\to L_p(\L\,G)$, we say that $T_m$ is a bounded $L_p$-multiplier.

The study of general conditions for the symbol $m$ that ensure the $L_p$-boundedness of $T_m$ has been an active area of research both in the classical and the non-commutative case. As discussed in the Introduction, the key result we are going to use concerns the following version of Cotlar identity for non-commutative Fourier multipliers:

\begin{theorem}\label{th: Cotlar theorem of Adri and Runlian}\cite[Theorem A]{GonParXia2022}
Let $G$ be a discrete group, $G_0 \leq G$ a subgroup and $\chi\colon G_0\to \mathbb{T}^1$ some character. Let $T_m$ be a Fourier multiplier whose symbol $m \colon G \to \mathbb{C}$
is bounded and left $(G_0,\chi)$-equivariant. If $m$ satisfies the identity:
\begin{equation*}
    (m(g^{-1}) - m(h)) (m(gh) - m(g)) = 0,
\text{ for all } g \in G \setminus G_0 \text{ and } h \in G
\end{equation*}
the $T_m$ is bounded in $L_p$ for all $1 < p < \infty$. Moreover, its norm satisfies:
\begin{equation*}
    \|T_m\colon L_p(\L\,G)\to L_p(\L\,G)\|\lesssim\left(\frac{p^2}{p-1}\right)^\beta,
    \quad \text{with } \beta = \log_2(1+\sqrt{2}).
\end{equation*}
\end{theorem}

The subgroup $G_0$ gives a range of flexibility to this result with respect to the original one of Cotlar: taking a big subgroup $G_0$ increases the chances for the formula to hold, but makes it harder for $m$ to satisfy the invariance hypothesis. 

\subsection*{Hyperbolic planes, their boundaries and M\"oebius transformations}
As we said, the group $\PSL_2(\mathbb{C})$ acts transitively on the set of pairs $(p,\Sigma)$ where $\Sigma$ is an hyperbolic plane embedded in $\mathbb{H}^3$ and $p$ is a point in $\Sigma$. When working with the upper-half space model $\mathbb{H}^3= \mathbb{C}+\mathbb{R}_{>0}j$, hyperbolic planes can be identified with half-planes and semispheres perpendicular (in the Euclidean sense) to $\mathbb{C}$. This induces a bijection between the set of hyperbolic planes in $\mathbb{H}^3$ and the set of generalized circles in $\widehat{\mathbb{C}} = \mathbb{C}\cup \{\infty\}$, that is, the set of lines and circles in $\mathbb{C}$.
We will denote by $\partial \Sigma$ the generalized circle associated to $\Sigma$ by this correspondence, and we will call it the boundary of $\Sigma$.

Notice also that, for any given hyperplane $\Sigma$ and any $g\in \PSL_2(\mathbb{C})$, it holds that $\partial(g\cdot \Sigma) = g\cdot \partial \Sigma$, where $g$ is acting by M\"obius transformation on the right-hand side. The action of Bianchi groups by M\"obius transformations have been extensively studied in \cite{Sta2017}. We will introduce now several results and concepts in that article, that we will make use of.

Let  $\extR = \mathbb{R}\cup \{\infty\}$ be the extended real line. The Bianchi group $\Gamma_n'$ acts on $\extR$ in a controlled way:
\begin{proposition}\cite[Proposition 4.4]{Sta2017}\label{prop: K-circles are tangent}
    If $n\neq 3$ and $g\in \Gamma_n'$, then $\extR$ and $g\cdot \extR$ may only intersect tangentially.
\end{proposition}

On the other hand, for a given $g\in \PGL_2(\mathbb{C})$ with $|\mathrm{det}(g)|=1$, the quantities:
\begin{equation}\label{eq: definition of curvature, cocurvature and curvature-center}
    \alpha = i(a\overline{d}-b\overline{c}) ,\quad \beta= -2\mathrm{Im}(c\overline{d}),\quad \beta'= -2\mathrm{Im}(a\overline{b}) ,\quad
    \text{where }
    g =
    \begin{bmatrix}
        a & b\\
        c & d
    \end{bmatrix},
\end{equation}
describe the image of $\extR$ by $g$ in the following way:
$$
g\cdot \extR = \left\{
X/Y\in \extC \colon \beta X\overline{X} - \alpha Y\overline{X} - \overline{\alpha}X\overline{Y} + \beta' Y\overline{Y} =0
\right\}.
$$
\begin{proposition}\cite[Propositions 3.5 and 3.7]{Sta2017}\label{prop: properties of centers}
    The coefficients $\alpha,$ $\beta$ and $\beta'$ defined as above verify that:
    \begin{enumerate}[label=\roman*.]
        \item\label{prop: properties of centers / equation for b,b',a} $\beta\beta' = |\alpha|^2 - 1$,
        \item the generalized circle $g\cdot \extR$ goes through $0$ if and only if $b'=0$,
        \item the generalized circle $g\cdot \extR$ is indeed a line if and only if $b =0$,
        \item if $b=0$, then $\alpha$ is a unit vector perpendicular to $g\cdot \extR$,
        \item if $b\neq 0$, then $g\cdot \extR$ is a circle of center $\alpha/\beta$ and radius $1/|\beta|$.
    \end{enumerate}
\end{proposition}

\section{Description of the set where Cotlar identity fails}\label{sec: classification of kernel}

Throughout the rest of the paper, we will denote by $\tau$ and $\tau'$ the following matrices:
\begin{equation*}
        \tau = \begin{bmatrix}
        i & 0 \\
        0 & 1
    \end{bmatrix}, \quad \text{and} \quad
        \tau' = \begin{bmatrix}
        0 & 1 \\
        -1 & 0
    \end{bmatrix}.
    \end{equation*}
Let $m$ be the function defined in \eqref{eq:def of m} and set $\Gamma_n = \PSL_2(\mathbb{Z}[\sqrt{-n}])$. As we shall prove later, our function $\restr{m}{\Gamma_n}$ is invariant (through a suitable character) with respect to: 
\begin{equation}\label{eq: def kernel of m}
    K_n = \{g\in \Gamma_n\colon m(g)=0\},
\end{equation}
which turns out to be a subgroup of $\Gamma_n$.
The goal of this section is to give an explicit description of this set. Along our proof, we will also give a description of the analogous set
\begin{equation}\label{eq: def kernel of m - Bianchi version}
    K_n' = \{g\in \Gamma_n' \colon m(g) = 0\}
\end{equation}
for $\Gamma_n'$ the Bianchi group of discriminant $-n$. These subsets $K'_n$ are defined only for square-free integers, and moreover $K'_n = K_n$ whenever $n \not\equiv -1 \; (\mathrm{mod}\, 4)$.

The main theorem of this section (namely, Theorem \ref{th: decomposition of kernel}) allows us to decompose $K_n$ and $K_n'$ as a combination of the four following disjoint sets:
\begin{equation}\label{eq: def of Kn+, etc}
\begin{split}
            K_n^+ &= \left\{
            \begin{bmatrix}
                x & y\sqrt{-n} \\
                z\sqrt{-n} & w
            \end{bmatrix}\colon x,y,z,w \in \mathbb{Z}, \; 
            xw + nyz = 1
            \right\},\\
        K_n^- &= \left\{
            \begin{bmatrix}
                x\sqrt{-n} & y \\
                z & w\sqrt{-n}
            \end{bmatrix}\colon x,y,z,w \in \mathbb{Z}, \; 
            nxw + yz = -1
            \right\},\\
        L_n^+ &= \left\{
            \begin{bmatrix}
                a & -\overline{a} \\
                c & \overline{c}
            \end{bmatrix}\colon a,c \in \O_{-n}, \; 
            \Re(a\overline{c})=\frac{1}{2}
            \right\},\quad \text{and}\\
        L_n^- &= \left\{
            \begin{bmatrix}
                a & \overline{a} \\
                c & -\overline{c}
            \end{bmatrix}\colon a,c \in \O_{-n}, \; 
            \Re(a\overline{c})=-\frac{1}{2}
            \right\}.    
\end{split}
\end{equation}

\begin{lemma}\label{lemma: for decompositions of kernels}
    Let $g\in \PSL_2(\mathbb{C})$, $G_0\leq \PSL_2(\mathbb{C})$ be the group defined in \eqref{eq: def of G0} and:
    \begin{equation*}
        L = \left\{
        \begin{bmatrix}
            a & -\overline{a} \\
            c & \overline{c}
        \end{bmatrix}\in \PSL_2(\mathbb{C}) \colon a,c\in \mathbb{C}, \, 2\mathrm{Re}(a\overline{c}) = 1
        \right\}.
    \end{equation*}
    Then it holds that:
    \begin{enumerate}[label=\roman*.]
        \item $g\cdot i\extR = i\extR$ if and only if $g\in G_0\cup \tau' G_0$,
        \item $g\cdot \mathbb{S}^1 = i\extR$ if and only if $g\in L\cup \tau' L$.
    \end{enumerate}
\end{lemma}
\begin{proof}
    Notice that $g\cdot i\extR = i\extR$ if and only if $\sigma(g)\cdot \extR = \extR$.
    It is well-known that:
    \begin{equation*}
        \{g\in \PSL_2(\mathbb{C})\colon g\cdot \extR =  \extR\} = \PSL_2(\mathbb{R}) \cup \begin{bmatrix}
        i & 0\\
        0 & -i
    \end{bmatrix}\PSL_2(\mathbb{R}).
    \end{equation*}
    The first point of the statement follows immediately.

    We claim now that any $g\in L$ verifies $g\cdot \mathbb{S}^1 = i\extR$. This is because $g^{-1}\cdot 0, \, g^{-1}\cdot \infty $ and $g^{-1}\cdot i$ can be very easily checked to be all complex numbers in $\mathbb{S}^1$. On the other hand, given any $g_0\in \PSL_2(\mathbb{C})$ such that $g_0\cdot \mathbb{S}^1 = i\extR$, the set of $g\in \PSL_2(\mathbb{C})$ satisfying $g\cdot \mathbb{S}^1 = i\extR$ decomposes as $G_0g_0\cup \tau' G_0g_0$. Since $G_0L \subset L$, the second point of the statement follows.
\end{proof}
\begin{ltheorem}\label{th: decomposition of kernel}
    Let $n\geq 1$ be an integer with $n\neq 3$, and $K_n$, $K_n'$ the sets defined in \eqref{eq: def kernel of m} and $\eqref{eq: def kernel of m - Bianchi version}$, respectively. Let also $\Sigma=\{\omega\in \mathbb{H}^3\colon \mathrm{Re}(\omega)=0\}$.
    Then, it holds that:
    \begin{enumerate}[label=\roman*.]
        \item $K_n$ is the stabilizer of $\Sigma$ in $\Gamma_n$.
            
        \item If $ n \equiv -1 \; (\mathrm{mod}\; 4)$ is a square free integer, then $K_n'$ is the union of the stabilizer of $\Sigma$ in $\Gamma_n'$, and the elements $g\in \Gamma_n'$ such that $ g\cdot \mathbb{S}^1 = i\extR$.
    \end{enumerate}
    Equivalently, $K_n= K_n^+ \cup K_n^-$ and $K_n' = K_n^+\cup K_n^-\cup L_n^+ \cup L_n^-$.
\end{ltheorem}

\begin{proof}
    Let $g$ be the matrix $$g=\begin{bmatrix}
            a & b \\
            c & d
        \end{bmatrix} \in \PSL_2(\mathbb{C}).$$ 
    Denote by $\sigma$ and $\sigma'$ the automorphisms of $\PSL_2(\mathbb{C})$ given by conjugation by $\tau$ and $\tau'$, respectively. Notice that $\sigma'(g^t)=g^{-1}$, so it follows that $g\cdot \Sigma = \Sigma$ if and only if $\sigma(g^t)\cdot \extR = \extR$, and $g\cdot \mathbb{S}^1=i\extR$ if and only if $\sigma(g^t)\cdot \extR = \mathbb{S}^1$. On the other hand, the quantities $\alpha$, $\beta$ and $\beta'$ defined in \eqref{eq: definition of curvature, cocurvature and curvature-center} for $\sigma(g^t)$ are the following:
    \begin{equation*}
        \alpha = i(a\overline{d} + c\overline{b}), \quad \beta = 2\mathrm{Re}(b\overline{d}), \quad \beta' = 2\mathrm{Re}(a\overline{c}).
    \end{equation*}
    It holds that $m(g)=0$ if and only if $\beta = \beta'$. Also, this implies that $|\alpha|^2 + |\beta|^2 = 1$. We consider now two cases:
    \begin{enumerate}[label=\roman*.]
        \item If $g\in \Gamma_n$, then $\beta,\beta'\in 2\mathbb{Z}$, so we conclude that $\beta = \beta' = 0$ and $|\alpha|=1$. Therefore, $\sigma(g^t)\cdot \extR$ is a line that goes through $0$ and has as orthogonal vector $\alpha$ (see Proposition \ref{prop: properties of centers}). Notice also that $\alpha\in i\mathbb{Z}[\sqrt{-n}]$. If $n>1$, then $\alpha\in \{-i,i\}$, so we get $\sigma(g^t)\cdot \extR = \extR$. If $n=1$, $\sigma(g^t)\in \Gamma_1$. By Proposition \ref{prop: K-circles are tangent} $,\sigma(g^t)\cdot \extR$ is tangent to $\extR$, so they must be the same line.

        \item If $g\in \Gamma_n'$ with $n \equiv -1 \; (\mathrm{mod}\, 4)$ and $n\neq 3$, then $\beta,\beta' \in \mathbb{Z}$ and $\alpha\in i\O_{-n}$. If $\beta = 0$, then $|\alpha|=1$ and therefore $\alpha\in \{i,-i\}$. This leads to $\sigma(g^t)\cdot \extR = \extR$ in the same fashion as before. On the other hand, if $|\beta|=1$ and $|\alpha|=0$, then $\sigma(g^t)\cdot \extR = \mathbb{S}^1$ by Proposition \ref{prop: properties of centers}.
    \end{enumerate}
    The rest of the statement follows from Lemma \ref{lemma: for decompositions of kernels}.
\end{proof}

\begin{remark}\label{remark: problem with Ln not being a group}
    Whereas $K_n$ is always a subgroup of $\Gamma_n$, $K_n'\subset \Gamma_n'$ is not, since it is not closed under products neither taking inverses. 
\end{remark}
\begin{remark}
    The theorem does not apply for $K_3'$. Notice that $\O_{-3}=\mathbb{Z}[\xi_3]$ where $\xi_3$ denotes a primitive $3$-root of the unit. A matrix as
    simple as:
    \begin{equation*}
        u=\begin{bmatrix}
        \xi_3 & 0\\
        0  & \overline{\xi_3}
    \end{bmatrix}
    \end{equation*}
    will be in $K_3'$ but not in $K_3\cup L_3$. Also, since $m$ is right $\PSU(2)$-equivariant, if we pick any $g\in K_3$ then $gu \in K_3'$, but this product will not be in $K_3\cup L_3$ in general.
\end{remark}

\section{Proof of the Cotlar Identity}

The sets $K_n^+$ and $K_n^-$ defined in \eqref{eq: def of Kn+, etc} verify certain relations related to the invariance of $m$:
$\tau' K_n^+ = K_n^+ \tau' = K_n^-.$
These identities, together with the fact that $K_n^+$ is a subgroup of $\Gamma_n$, implies easily that:
\begin{equation*}
     K_n^+K_n^-, \, K_n^-K_n^+ \subset K_n^- 
     \quad \text{and}\quad
     K_n^-K_n^- \subset K_n^+.
\end{equation*}
 We claim now that, because of these inclusions, the function $\chi\colon K_n\to \mathbb{T}^1$ defined as:
\begin{equation*}
\chi(g)=
\begin{cases}
1 & \text{if } g\in K_n^+,\\
-1 & \text{if } g\in K_n^-,
\end{cases}
\end{equation*}
is a character. The following three lemmas prove that $\restr{m}{\Gamma_n}$ is left $(K_n,\chi)$-equivariant.

\begin{lemma}\label{lemma: ANK decomposition}
    Let $g\in \mathrm{PSL}_2(\mathbb{C})$ and let $r_1(g)$ and $r_2(g)$ denote the first and second rows of $g$, respectively. There exist an unitary matrix $u\in \mathrm{PSU}(2)$ such that:
    \begin{equation*}
        g =
        \begin{bmatrix}
            s^{-1} & s^{-1}t \\
            0 & s
        \end{bmatrix}u,
    \end{equation*}
    with $s = |r_2(g)|$ and $t = \langle r_1(g), r_2(g) \rangle$, where the bracket represents the scalar product in $\mathbb{C}^2$.
\end{lemma}
\begin{proof}
    This is just an explicit statement of the $\mathrm{ANK}$ decomposition for $\PSL_2(\mathbb{C})$. It can be proven directly as follows. Let $u$ be the (only) unitary matrix such that $r_2(g)u^* = (0, s)$ with $s>0$. Thus, $s = |r_2(g)|$. On the other hand, using that $\det (gu^*)=1$, we get that $r_1(g)u^* = (s^{-1}, \omega)$ for some $\omega\in \mathbb{C}$. This $\omega$ can be computed using that $\omega = s^{-1}\langle r_1(gu^*), r_2(gu^*) \rangle = s^{-1}\langle r_1(g), r_2(g) \rangle$, which is the definition of $s^{-1}t$.
\end{proof}

\begin{lemma}\label{lemma: quadratic inequality}
    For any $g = 
    \begin{bmatrix}
        a & b \\
        c & d
    \end{bmatrix} \in \PSL_2(\mathbb{C})$, it holds that:
    $$ \Im(b\overline{c} - a\overline{d})^2 - 4 \Re(a\overline{c})\Re(b\overline{d}) \leq 1.
    $$
    Moreover, if $g\in \Gamma_n$, then the right-hand side of the above inequality can be improved to $0$.
\end{lemma}
\begin{proof}
    Same computations as in the proof of \cite[Lemma 5.3]{GonParXia2022}
    shows that the left-hand side of the above expression can be written as $p(X)= -4X(1+X)$, where $X = na_2d_2+b_1c_1$. This proves the statement for $g\in \PSL_2(\mathbb{C})$. If $g\in \Gamma_n$, then $X$ is an integer and therefore $p(X)\in 4\mathbb{Z}$, which proves the second part of the statement. 
\end{proof}
\begin{lemma}\label{lemma: invariance of m}
    The symbol $\restr{m}{\Gamma_n}$ is right $K_n$-invariant and left $(K_n,\chi)$-equivariant.
\end{lemma}
\begin{proof}
    It is immediate to check that 
        $m(\omega g) = - m(g).$
    On the other hand, $K_n^+$ and $K_n^-$ are contained respectively in $G_0$ and $\omega G_0$, where $G_0$ is the group defined in the Introduction by \eqref{eq: def of G0}. Since $m$ is invariant by the left action of $G_0$,
    it follows that $\restr{m}{\Gamma_n}$ is left $(K_n,\chi)$-equivariant.

     For the right invariance, let's take $g\in \Gamma_n$ and $h\in K_n$. If $g\in K_n$, it is immediate that $m(gh)=m(g)=0$, so we rule out this case. Let's write $g$ and $h$ as
     \begin{equation*}
         g = \begin{bmatrix}
             a & b \\
             c & d
         \end{bmatrix}
         \quad \text{and} \quad
         h = \begin{bmatrix}
             s^{-1} & s^{-1}t\\
             0 & s
         \end{bmatrix}u,
     \end{equation*}
     where we used Lemma \ref{lemma: ANK decomposition} to decompose $h$ in a product of two matrices, such that $u\in \PSU(2)$, $s>0$ and $t = \langle r_1(h), r_2(h)\rangle$. Recall that $r_1$ and $r_2$ represent the first and second rows of our matrices, and $\langle \cdot, \cdot \rangle$ is the scalar product in $\mathbb{C}^2$. Since $h\in K_n$, $t$ is purely imaginary, which allows us to write:
     \begin{equation}\label{eq: matrix equation of quadratic ineq}
     \begin{split}
         \Re \langle r_1(gh), r_2(gh) \rangle 
         &= 
         \Re(a\overline{c})(1+(\Im t)^2)s^{-2}
         + \Re(b\overline{d})s^2
         + \Im(b\overline{c}-a\overline{d})\Im t \\
         &=
         \begin{bmatrix}
             s^{-1}\Im t & s
         \end{bmatrix}
         \begin{bmatrix}
             2 \Re (a\overline{c}) & \Im(b\overline{c}-a\overline{d}) \\
             \Im(b\overline{c}-a\overline{d}) & 2 \Re (b\overline{d})
         \end{bmatrix}
         \begin{bmatrix}
             s^{-1}\Im t \\ s
         \end{bmatrix} \\ & \qquad + s^{-2}\Re(a\overline{c})
     \end{split}
     \end{equation}
     The Lemma \ref{lemma: quadratic inequality} says that the determinant of the matrix in \eqref{eq: matrix equation of quadratic ineq} is always non-negative. Therefore, this matrix will be semidefinite positive if $\Re(a\overline{c})\geq 0$ and $\Re(b\overline{d})\geq 0$ and semidefinte negative otherwise. In both cases, it implies that $m(gh)m(g)\geq 0$. Since $gh\not \in K_n$, it follows that $m(gh)=m(g)$, proving the statement.
\end{proof}

In the proof of the Theorem \ref{th:main theorem}, we will make use of two inequalities that we introduce now as two independent lemmas.
\begin{lemma}\label{lemma: relation m(g)m(gt)}
    Let $g = 
    \begin{bmatrix}
        a & b \\
        c & d
    \end{bmatrix}\in \Gamma_n$. Then 
    \begin{equation*}
    m(g)m(g^t)\Re(a\overline{d}+b\overline{c})\geq 0,
    \end{equation*}
    where $g^t$ denotes the transpose of $g$.
\end{lemma}
\begin{proof}
    If $g$ or $g^t$ are in $K_n$, the result is immediate. If they are not, we know $m(g)$ has the same sign as $\Re(a\overline{c})$ or $\Re(b\overline{d})$, depending on which one is non-zero. We'll suppose that both $\Re(a\overline{c})$ and $\Re(a\overline{b})$ are non-zero, since the rest of the cases comes from applying this one to $\tau' g$, $g\tau'$ or $\tau' g \tau'$.

Under this hypothesis, the statement is equivalent to
    \begin{equation*}
    \Re(a\overline{c})\Re(a\overline{d})\Re(a\overline{d}+b\overline{c})\geq 0.
    \end{equation*}
     From the proof of \cite[Proposition 5.8]{GonParXia2022} we know that the left-hand side of the inequality equals $p(X)=(AX+B)(2X+1)$ with $A = n(a_1^2+a_2^2)$, $B=na_2^2$ and $X = b_1c_1+na_2d_2$. Since $X$ is an integer and the roots of the polynomial $p$ have modulus lesser or equal than $1$, we conclude the statement.
\end{proof}

\begin{lemma}\label{lemma: lower bound of scalar product column by column}
    For any 
    $g = \begin{bmatrix}
        a & b \\
        c & d
    \end{bmatrix}\in \PSL_2(\mathbb{C})$, it holds that $\Re(a\overline{c})\Re (b \overline d)\geq -\frac{1}{4}$.
    Moreover, if $g\in \Gamma_n$, the right-hand side of the inequality can be improved to $0$.
\end{lemma}
\begin{proof}
    Suppose that $\Re(a\overline{c})\Re(b\overline{d})<0$. Then, by multiplying the equation $ad-bc=1$ by $\overline{cd}$ and taking real part, we get that:
    \begin{equation*}
        |\Re(b\overline{d})||c|^2 + |\Re(a\overline{c})||d|^2 = |\Re(c\overline{d})| \leq |c||d|.
    \end{equation*}
    Now we claim that any positive numbers $x,y,\alpha,\beta >0$ satisfying 
    \begin{equation}\label{eq: reverse Cauchy with weights}
        \alpha x^2 + \beta y^2 \leq xy
    \end{equation}
    must verify $\alpha\beta\leq \frac{1}{4}$. To prove the claim, just notice that \eqref{eq: reverse Cauchy with weights} is equivalent to $\alpha u^2 - u + \beta \leq 0$ with $u=x/y$, and this can only happen if the discriminant $1-4 \alpha \beta$ is greater than or equal to $0$.

    If $g\in \Gamma_n$, then both $\Re(a\overline{c})$ and $\Re(b\overline{d})$ are integers, so $\Re(a\overline{c})\Re (b \overline d)$ must be indeed non-negative.
\end{proof}

\begin{proof}[Proof of Theorem \ref{th:main theorem}]
We are going to prove that the symbol $\restr{m}{\Gamma_n}$ satisfies \eqref{eq:Cotlar for symbol} relative to $K_n$, that is:
\begin{equation*}
    (m(g^{-1}) - m(h)) (m(gh) - m(g)) = 0,
\text{ for all } g \in \Gamma_n \setminus K_n \text{ and } h \in \Gamma_n.
\end{equation*}
If $h\in K_n$, the equality follows from the right $K_n$-invariance of $m$ proven in Lemma \ref{lemma: invariance of m}. Now, suppose that $h\not\in K_n$ and $m(g^{-1})\neq m(h)$. We have to prove that $m(gh)=m(g)$. Since the hypothesis $m(g^{-1})\neq m(h)$ implies that $gh\not\in K_n$, it suffies to prove that $m(gh)m(g)\geq 0$. We write:
\begin{equation*}
         g = \begin{bmatrix}
             a & b \\
             c & d
         \end{bmatrix}
         \quad \text{and} \quad
         h = \begin{bmatrix}
             s^{-1} & s^{-1}t\\
             0 & s
         \end{bmatrix}u,
     \end{equation*}
using Lemma \ref{lemma: ANK decomposition} to decompose $h$ into an upper-triangular matrix and an unitary one. Now, a computation shows that:
\begin{equation*}
\begin{split}
        m(gh)m(g)
        &=
        \operatorname{sign}\Big( \Re(a\overline{c}+b\overline{d})\Re(a\overline{c})s^{-2}(1+(\Re t)^2)\\
        &\qquad +
         \Re(a \overline{c}+b\overline{d})\Re(a\overline{d}+b\overline{c})\Re t\\
        &\qquad +
        \Re(a\overline{c}+b\overline{d})\big[\Re(a\overline{c})s^{-2}(\Im t)^2 
            + \Re(b\overline{d})s^{2}
            + \Im(b\overline{c} - a\overline{d})\Im t\big] \Big) \\
        &=
        \operatorname{sign}\Big((\mathrm{I}) + (\mathrm{II}) +(\mathrm{III})\Big)
\end{split}
\end{equation*}
Now notice that $(\mathrm{I})$ is non-negative because of Lemma \ref{lemma: lower bound of scalar product column by column} and the fact that $s>0$. Also, $(\mathrm{II})$ is non-negative because $\Re t$ has the same sign as $m(h) = -m(g^{-1}) = m(g^t)$, so we can apply Lemma \ref{lemma: relation m(g)m(gt)}. Finally, $(\mathrm{III})$ is non-negative because of Lemma \ref{lemma: quadratic inequality}, which implies that each factor of the product has the same sign as $m(g)$ or is zero.
\end{proof}

\begin{remark}
    We still don't know if the Fourier multiplier given by $\restr{m}{\Gamma_n'}$ is bounded or not in $L_p(\L\, \Gamma_n')$, but what can be proven is that this symbol do not verify a Cotlar identity as in Theorem \ref{th: Cotlar theorem of Adri and Runlian} with respect to any possible subgroup of $\Gamma_n'$. To see this, suppose that $\restr{m}{\Gamma_n'}$ is $(G_0,\chi)$-equivariant for some subgroup $G_0\leq \Gamma_n'$ and some character $\chi$ on $G_0$. We claim that $G_0\cap L_n^+= \varnothing$. 
    Firstly, notice that for any $l\in L_n^+$ and $g\in \PSL_2(\mathbb{C})$, it holds that:
    \begin{equation*}
        m(lg) = \operatorname{sign} \left(|r_1(g)|-|r_2(g)|\right)
    \end{equation*}
    where $r_1(g)$ and $r_2(g)$ denotes the first and second rows of $g$ as complex vectors in $\mathbb{C}^2$.
    Let $a\colon \Gamma_n'\to \Gamma_n'$ be the map that permutes the two rows of a matrix and multiplies the first column by $-1$.
    If $G_0\cap L_n^+\neq \varnothing$, by the formula above it would hold that for any $h\in \Gamma_n'$ and any $l\in G_0\cap L_n^+$:
    \begin{equation*}
        m(lh) = \chi(l)m(h) = \chi(l)m( a(h)) = m(l a(h) ) = - m(lh),
    \end{equation*}
    which is of course impossible. On the other hand, fix an $l\in L^+_n$ whose inverse is not in $K_n'$. Then in order to Cotlar identity to hold, one needs that:
    \begin{equation*}\label{eq: Cotlar in L+ must fail}
        m(l^{-1}) = m(h), \quad
        \text{for any } h\in \Gamma_n' \text{ such that } m(lh) \neq 0.
    \end{equation*}
    Pick $h\in \Gamma_n'$ any element which verifies this equation. Let $h'$ be given by $h' = \tau' h \tau'$, where $\tau'$ is the matrix defined at the beginning of Section \ref{sec: classification of kernel}. Notice that $m(lh') = - m(lh)\neq 0$, but $m(h') = - m(h)$. Therefore Cotlar identity must fail when applied to $l$ and $h'$.
\end{remark}

\section{Addenda: another \texorpdfstring{$L_p$}{Lp}-bounded multiplier on Bianchi groups}

Our initial choice of hyperbolic plane $\Sigma = \{\omega \in \mathbb{H}^3\colon \mathrm{Re}(\omega) = 0\}$ was motivated by the authors of \cite{GonParXia2022} proving that the corresponding multiplier $m$ satisfies the Cotlar identity in $\mathrm{SL}_2(\mathbb{Z})$ and $\mathrm{SL}_2(\mathbb{Z}[i])$. However, we have seen that such an identity fails for $m$ when restricted to a general Bianchi group $\Gamma_n'$. This failure is connected to the geometry of circles in the orbit of $\partial \Sigma$ under the action of $\Gamma_n'$, so 
it is natural to ask for multipliers induced by planes with a better-behaved boundary.
Concretely, we are going to consider now the multiplier $\widetilde m$ induced by the hyperplane $\widetilde \Sigma = \{\omega \in \mathbb{H}^3\colon \mathrm{Im}(\omega) = 0\}$, and the partition $\widetilde \Sigma^+ = \{\omega \in \mathbb{H}^3\colon \mathrm{Im}(\omega) > 0\}$, $\widetilde \Sigma^- = \{\omega \in \mathbb{H}^3\colon \mathrm{Im}(\omega) < 0\}$.

Indeed, we claim that  $m$ will satisfy the Cotlar identity on $\mathrm{SL}_2(\mathbb{Z}[i])$ if and only if $\widetilde m$ does so. Let $\sigma$ be the automorphism of $\PSL_2(\mathbb{C})$ given by conjugation by $\tau$, where $\tau$ is the matrix defined at the beginning of Section \ref{sec: classification of kernel}.
A simple computation shows that:
\begin{align*}
    \widetilde m(\sigma(g)) &= \mathrm{sign} \, \mathrm{Im}\left[ 
    (aj+bi)\overline{(-cij + d)}
    \right] \\
    &= \mathrm{sign} \, \mathrm{Re}\left( 
    a\overline{c} + b\overline{d}
    \right) = m(g).
\end{align*}

The automorphism $\sigma$ leaves $\mathrm{SL}_2(\mathbb{Z}[i])$ invariant, so $\sigma$ restricts to an automorphism of this group, proving our claim. Therefore, this point of view also generalizes the results of \cite{GonParXia2022} in a different way than we did before.

\begin{lemma}\label{lemma: kernel for the second multiplier}
    For any $n\geq 1$ with $n\neq 3$, the set $\widetilde{K}_n = \{g\in \Gamma_n'\colon \widetilde{m} (g) = 0\}$ coincides with the stabilizer of $\widetilde \Sigma$ under the action of $\Gamma_n'$. Indeed, $\widetilde m (g)$ can be written as:
    \begin{equation*}
        \widetilde m (g) =
        \begin{cases}
            1 & \text{if } g\cdot \widehat{\mathbb{R}} \text{ lies on the upper half-plane},\\
            -1 & \text{if } g\cdot \widehat{\mathbb{R}} \text{ lies on the lower half-plane},\\
            0 & \text{otherwise}.
        \end{cases}
    \end{equation*}
\end{lemma}
\begin{proof}
    Notice first that two hyperbolic planes $\Sigma_1$ and $\Sigma_2$ verify $\Sigma_1\cap \Sigma_2 \neq \emptyset$  if and only if $\partial \Sigma_1$ and $\partial\Sigma_2$ intersect in at least two points.
    If $g\in \Gamma_n'$ and $\widetilde{m}(g)=0$, it is because $g\cdot \widetilde \Sigma \cap \widetilde \Sigma \neq \emptyset$. Since $\partial \widetilde \Sigma = \widehat{\mathbb{R}}$, this implies that $g\cdot \widehat{\mathbb{R}}$ and $\widehat{\mathbb{R}}$ intersect in at least two points, but because of Proposition \ref{prop: K-circles are tangent} it means that they are the same generalized circle. Therefore $g\cdot \widehat{\mathbb{R}} = \widehat{\mathbb{R}}$ and $g\cdot \widetilde \Sigma = \widetilde \Sigma$. The rest of the statement follows similarly.
\end{proof}

\begin{lemma}\label{lemma: invariance for the second multiplier}
    The multiplier $\widetilde m$ restricted to $\Gamma_n'$ is right $\widetilde K_n$-invariant and left $(\widetilde K_n, \chi)$-equivariant, for some character $\chi$ on $\widetilde K_n$.
\end{lemma}
\begin{proof}
    Both the invariance and the equivariance follow easily from Lemma \ref{lemma: kernel for the second multiplier}. just taking the character $\chi$ defined as:
    \begin{equation*}
        \chi(g)=
        \begin{cases}
            1 & \text{if } g \text{ send the upper half-plane to itself},\\
            -1 & \text{otherwise}.
        \end{cases}
    \end{equation*}
\end{proof}

\begin{theorem}
    The symbol $\widetilde m$ satisfies the Cotlar identity on every Bianchi group $\Gamma_n'$ with $n\neq 3$. 
\end{theorem}
\begin{proof}
    Because of Lemma \ref{lemma: invariance for the second multiplier}, it is enough to prove that
    \begin{equation*}
    (\widetilde m(g^{-1}) - \widetilde m(h)) (\widetilde m(gh) - \widetilde m(g)) = 0,
\text{ for all } g, h \in \Gamma_n' \smallsetminus \widetilde K_n.
\end{equation*}
    Let's suppose that $\widetilde m(g^{-1})\neq \widetilde m(h)$. Then, without lost of generality, we can suppose that there is a generalized circle $C_g$ in the upper half-plane and a generalized circle $C_h$ in the lower half-plane such that $g\cdot C_g = \widehat{\mathbb{R}}$ and $h\cdot \widehat{\mathbb{R}} = C_h$. Since $\widehat{\mathbb{R}}$ and $C_h$ lie both on the exterior of $C_g$, they are mapped into the same half-plane by $g$. That is, $g\cdot C_h = gh\cdot \widehat{\mathbb{R}}$ and $g\cdot\widehat{\mathbb{R}}$ lie on the same side of $\widehat{\mathbb{C}}\smallsetminus\widehat{\mathbb{R}}$, and therefore $\widetilde m (gh) = \widetilde m(g)$. 
\end{proof}
\subsection*{Acknowledgments} The author was partially supported by pre-doctoral scholarship PRE2020-093245, 
    the Severo Ochoa Grant CEX2023-001347-5(MICIU), \linebreak 
    PIE2023-50E106(CSIC), 
    and PID2022-141354NB-I00(MICINN). He would also like to thank his advisor, Adrián González-Pérez, for pointing him out the problem, as well as for his support and insightful conversations.

\end{document}

%% file: abstract.tex
Gonz\'alez-P\'erez, Parcet and Xia introduced recently a framework to study $L_p$-boundedness of certain families of idempotent multipliers on von Neumann algebras. It includes symbols $m\colon \mathrm{PSL}_2(\mathbb{C})\to \mathbb{R}$ arising from lifting the indicator function of a partition $\{\Sigma^+,\Sigma^+,\Sigma^-\}$ of the hyperbolic space $\mathbb{H}^3$ to its isometry group $\mathrm{PSL}_2(\mathbb{C})$. The boundedness of $T_m$ on $L_p(\mathcal{L} \mathrm{PSL}_2(\mathbb{C}))$ was disproved by Parcet, de la Salle and Tablate. Nevertheless, we will show that this Fourier multiplier is bounded when restricted to the arithmetic lattices $\mathrm{PSL}_2(\mathbb{Z}[\sqrt{-n}])$, solving a question left open by the first named authors.